\documentclass[letterpaper, 10 pt, conference]{ieeeconf}  

\IEEEoverridecommandlockouts                              
\usepackage{graphics} 
\usepackage{epsfig} 
\usepackage{tikz,pmat,cite}
\usepackage{amsmath} 
\usepackage{amssymb}  
\usepackage{url}
\usepackage[export]{adjustbox}
\usepackage{subcaption}

\graphicspath{{./Figures/}}

\newcommand{\x}{\mathbf{x}}

\newcommand{\uu}{\mathbf{u}}
\newcommand{\yy}{\mathbf{y}}
\newcommand{\f}{\mathbf{f}}
\newcommand{\g}{\mathbf{g}}
\newcommand{\A}{\mathbf{A}}
\newcommand{\B}{\mathbf{B}}
\newcommand{\C}{\mathbf{C}}
\newcommand{\D}{\mathbf{D}}
\newcommand{\M}{\mathbf{M}}

\newcommand{\X}{\mathcal{X}}
\newcommand{\G}{\mathcal{G}}
\newcommand{\U}{\mathcal{U}}
\newcommand{\E}{\mathcal{E}}
\newcommand{\R}{\mathbb{R}}
\newcommand{\ETA}{\pmb{\eta}}

\newcommand{\NU}{\pmb{\nu}}
\newtheorem{theorem}{Theorem}
\newtheorem{definition}{Definition}
\newtheorem{assumption}{A\hspace{-0.1cm}}

\newtheorem{remark}{Remark}
\newtheorem{proposition}{Proposition}

\newtheorem{problem}{Problem}

\title{\LARGE \bf
Guaranteed Fault Detection and Isolation for Switched Affine Models}

\author{Farshad Harirchi\;\;\;\; Sze Zheng Yong\;\;\;\; Necmiye Ozay
	\thanks{This work is supported in part by DARPA grant N66001-14-1-4045 and an 
Early Career Faculty grant from NASA's Space Technology Research Grants Program.}
	\thanks{F. Harirchi and N. Ozay are with the Electrical Engineering and Computer Science Department, University of Michigan, Ann Arbor, MI, 48109 (email: {\tt {\{harirchi,necmiye\}@umich.edu}});
	S.Z. Yong is with the School for Engineering in Matter, Transport and Energy, Arizona State University, Tempe, AZ, 85281 (email: {\tt {szyong@asu.edu}}).}\vspace{-0.cm}}

\begin{document}

\maketitle
\thispagestyle{empty}
\pagestyle{empty}
\begin{abstract}
This paper considers the problem of fault detection and isolation (FDI) for switched affine models. We first study the model invalidation problem and its application to guaranteed fault detection. Novel and intuitive optimization-based formulations are proposed for model invalidation and $T$-distinguishability problems, which we demonstrate to be computationally more efficient than an earlier formulation that required a complicated change of variables. 
Moreover, we introduce a \emph{distinguishability index} as a measure of separation between the system and fault models, which offers a practical method for finding the smallest receding time horizon that is required for fault detection, and for finding potential design recommendations for ensuring $T$-distinguishability. 
Then, we extend our fault detection guarantees to the problem of fault isolation with multiple fault models, i.e., the identification of the type and location of faults, by introducing the concept of $I$-isolability. 
An efficient way to implement the FDI scheme is also proposed, whose run-time  does not grow with the number of fault models that are considered. 
Moreover, we derive bounds on detection and isolation delays and present an adaptive 
scheme for reducing isolation delays. 
Finally, the effectiveness of the proposed method is illustrated using several examples, including an HVAC system model with multiple faults.
\end{abstract}

\section{Introduction} \label{sec:intro}
Cyber-physical systems (CPS), i.e., systems with integrated computation, networking, and physical processes, are becoming increasingly common in our daily lives. Such systems include critical infrastructures such as traffic, power and water networks, as well as autonomous vehicles, aircrafts, home appliances and manufacturing processes. 
However, some major incidents involving these critical infrastructure systems as a result of cyber-attacks and system failures have taken place in the recent years and are a big source of concern. Hence, the reliability and security of CPS is paramount for their successful implementation and operation.  The detection and isolation of faults and anomalies in CPS play an important role in enhancing the reliability of these systems, and in understanding the vulnerability of system components to failures and attacks.


\subsubsection{Literature Review}
The study of fault detection began with the introduction of the first failure detection filter by Beard in 1971 \cite{Beard1971Failure}. Since then, fault diagnosis has attracted a great deal of attention and has become an integral part of most, if not all system designs. Researchers have mainly approached the fault detection and isolation problem by employing either data-driven techniques or model-based approaches. These methods can, in general, be further grouped into active and passive approaches. In active fault detection, the system is excited by a carefully designed input \cite{Nikoukhah1998Guaranteed,Nikoukhah2006Auxiliary,Harirchi2017active}, while in passive methods, the behavior of the system is not controlled or altered by the FDI scheme \cite{Frank1997Survey,Patton1997Observer}. 

In broad strokes, model-based fault detection and isolation schemes in the literature can be categorized into two classes, i.e., approaches that are based on residual generation and on set-membership. The former approach is more common in the fault diagnosis literature, and in this approach, the difference between the measurements and the estimates is defined as a residual or a symptom \cite{Simani2003Model}. Two major trends in the residual generation techniques are methods based on observers 
\cite{Frank1993Advances,Frank1997Survey,Sneider1996Observer,Hammouri1999Observer,  Paoletti2008Necessary, Abdo2012Fault,  Pan2015Online,Narasimhan2007Model} and parameter estimation  \cite{Isermann1993Fault,Liu2000Fault}.

On the other hand, set-membership based fault detection and isolation techniques are proposed with the goal of 
providing guarantees for the detection of some specific faults. Most of these methods operate by discarding models that are not compatible with observed data, rather than identifying the most likely model. There is an extensive literature on set-membership based methods for active fault detection of linear models \cite{Campbell2004Auxiliary, Scott2014Input, Rosa2010Fault}. In \cite{Harirchi2015Model,Harirchi2016Model}, we posed set-membership based guaranteed passive fault detection approaches for the class of switched affine models and polynomial state space models. These approaches are developed by utilizing ideas from model invalidation \cite{Anderson2009Validation,OzayConvex2014} and taking advantage of recent advances in optimization. In addition, a concept called $T$-distinguishability has been introduced for finding conditions under which the fault detection scheme can be applied in a receding horizon manner without compromising detection guarantees. $T$-distinguishability is closely related to the concept of input-distinguishability of linear systems \cite{Lou2009Distinguishability,Rosa2011Distinguishability} and mode discernibility in hybrid systems \cite{Babaali2004Observability}.

\subsubsection{Main Contributions and Paper Structure}
In this paper, we consider a passive fault detection and isolation scheme for  switched affine systems using an optimization-based model invalidation framework, that improves and expands the results of \cite{Harirchi2015Model} on fault detection. We provide novel formulations of the model invalidation and $T$-distinguishability algorithms that we demonstrate to be noticeably faster than the previous formulation in \cite{Harirchi2015Model} and that have the added advantage of simplicity as no complicated change of variables are needed. Furthermore, we introduce a measure of separation between models, called \emph{distinguishability index}. By reformulating the $T$-distinguishability optimization problem in \cite{Harirchi2015Model}, we can compute the distinguishability index as a byproduct. 
This index offers a practical way to find out if a finite receding time horizon exists, and suggests potential design options for ensuring $T$-distinguishability. 

We then consider the fault isolation problem using model invalidation and introduce the concept of $I$-isolability when multiple faults are present. Similar to fault detection, we propose a computationally efficient optimization problem to check whether a given set of fault models is $I$-isolable or not. 
Moreover, we propose a fault diagnosis scheme that not only detects the occurrence of a fault, but also outputs a list of potential faults along with their associated `likelihoods' in the form of distinguishability indices. Further, a theoretical analysis for bounds on detection and isolation delays is provided and an adaptive fault isolation scheme is proposed to reduce isolation delays. The run-time of our FDI scheme also does not grow with the number of fault models. Finally, these results are illustrated using a numerical model of a heating, ventilating, and air conditioning (HVAC) system.


\section{Preliminaries} \label{sec:prelim}
In this section, the notation used throughout the paper and the modeling framework we consider are described.

\subsection{Notation}
Let $ \x \in \R^n$ denote a vector and $ \M \in \R^{n\times m}$ represent a matrix. The infinity norm of a vector $\textbf{x}$ is denoted by $\| \textbf{x} \| \doteq \max_i |\x^i|$, where $\x^i$ denotes the $i^{th}$ element of vector $\x$.  The set of positive integers up to $n$ is denoted by $\mathbb{Z}_n^{+} $, and the set of non-negative integers up to $n$ is denoted by $\mathbb{Z}_n^{0}$.

\subsection{Modeling Framework}
In this paper, we consider systems that can be represented by discrete-time switched affine (SWA) models. 

\begin{definition}
	(SWA Model) A switched affine model is defined by:
	\begin{equation} \label{eqn:USWA}
	\mathcal{G} = (\X, \E, \U , \{ G_i \}_{i=1}^m),
	\end{equation} 
	where $\X \subset \R^n$ is the set of states, $\E \subset \R^{n_y+n_p}$ is the set of measurement and process noise signals, $\mathcal{U} \subset \R^{n_u}$ is the set of inputs and $\{G_i\}_{i=1}^m$ is a collection of $m$  modes. For all $i\in \mathbb{Z}_{m}^+$, the $i^{th}$ mode is an affine model:
	\begin{align}
	G_i = \{ \A_i,  \B_i, \C_i, \D_i, \f_i,  \g_i \}.
	\end{align}
	The evolution of $\G$ is governed by:
	\begin{equation} \label{eqn:SWA}
	\begin{aligned} 
	\x_{t+1}  	& = \mathbf{A}_{\sigma_t} \x_t + \B_{\sigma_t} \uu_t + \f_{\sigma_t} + \NU_t,\\
	\mathbf{y}_t  	& = \C_{\sigma_t} \x_t+ \mathbf{D}_{\sigma_t} \mathbf{u}_t+ \g_{\sigma_t} +\ETA_t,
	\end{aligned}
	\end{equation}
	where $\NU \in \R^{n_p}$ and $\ETA \in \R^{n_y}$ denote the process and measurement noise signals, respectively, and $\sigma_t$ indicates the active mode at time $t$.
\end{definition}
\begin{remark}
	We assume $\X, \E, \U$ are convex and compact sets. In particular, we consider the following form for the admissible sets:
	\begin{equation}
	\begin{aligned}
	& \X = \{ \x \mid P\x \leq p \}, \; \E \hspace{-0.05cm}=\hspace{-0.05cm} \{ [\ETA^\intercal \; \; \NU^\intercal]^\intercal \mid \|\ETA\| \leq \epsilon_{\eta}, \|\NU\| \leq \epsilon_{\nu} \},\\
	& \U = \{ \uu \mid \|\uu\| \leq U \}, 
	\end{aligned}
	\end{equation}
	where $P \in \R^{n_p \times n}$ and $p \in \R^{n_p}$. Note that our analysis holds true for any $\X, \E, \U$ that are convex sets, but for 
	simplicity in notation, we use the above mentioned admissible sets.
\end{remark}

We define the fault model as follows:
\begin{definition}[Fault Model] \label{def:fault_mod} A \emph{fault model} for a switched affine system $\G= (\X, \E, \U ,  \{ G_i\}_{i=1}^m)$ is another switched affine model $\bar{\G}= (\bar{\X}, \bar{\E}, \bar{\U} , \{ \bar{G}_i\}_{i=1}^{\bar{m}})$ with the same number of states, inputs and outputs. 
\end{definition}

Further, to describe our framework of model invalidation and $T$-distinguishability for fault detection and isolation in the next section, we define the following.
\begin{definition}[Length-$N$ behavior] The \emph{length-$N$ behavior} associated with an SWA system $\G$ is the set of all length-$N$ input-output trajectories compatible with $\G$, given by the following set:\vspace{-0.2cm}
	
\small{ 
\begin{equation*}
	\begin{aligned}
	\mathcal{B}^N_{swa}(\G) := \big \{ \{\uu_t, \yy_t\}_{t=0}^{N-1} \mid \uu_t \in \mathcal{U}  \text{ and } \exists \x_t\in \mathcal{X}, \sigma_t\in \mathbb{Z}_m^+, \\ 
	[\ETA_t^\intercal \; \; \NU_t^\intercal]^\intercal \in \mathcal{E}, \text{ for } t=0,\ldots, N-1 \text{ s.t. } \eqref{eqn:SWA} \text{ holds} \big \}.
	\end{aligned} 
\end{equation*} }
\normalsize
\hspace{-0.15cm}Moreover, with a slight abuse of terminology, we will call $\mathcal{B}^N_{swa}(\G)$ the \emph{behavior} of the system $\G$ for conciseness.
\end{definition} 

\section{Model Invalidation and $T$-Detectability} \label{sec:modelInv}

\subsection{Model Invalidation}
 In our previous work \cite{Harirchi2015Model,Harirchi2016Model}, we established a theoretical framework that can be utilized in order to develop fault detection schemes based on the achievements in model invalidation, a framework that we will also consider in this paper. The model invalidation problem is to check whether some given data can be represented by a model or not. More formally, the model invalidation problem is as follows: 
\begin{problem}[Model Invalidation] \label{Pro:MIP}
	Given an SWA model $\G$ and an input-output sequence $\big \{\uu_t, \yy_t \big \}_{t=0}^{N-1}$, determine whether or not the input-output sequence is contained in the behavior of $\G$, i.e., whether or not the following is true:
	\begin{align} \label{eqn:MIPDEF}
	\big \{\uu_t,\yy_t \big \}_{t=0}^{N-1} \in \mathcal{B}_{swa}^N(\G).
	\end{align}
\end{problem}

Clearly, if the model is invalidated by data, i.e., \eqref{eqn:MIPDEF} does not hold, and the model is precise, it is equivalent to the data representing an abnormal behavior. Hence, model invalidation can be used as a fault detection scheme. Moreover, our previous work \cite{Harirchi2015Model} has shown that model invalidation problem for SWA models can be posed as a Mixed-Integer Linear Program (MILP) feasibility check problem. 

In this paper, we propose a new MILP formulation, that we believe is much more intuitive and computationally efficient. We obtain this novel formulation by taking advantage of Special Ordered Set of degree 1 (SOS-1) constraints \cite{Beale1976Global}, that are readily implementable in most off-the-shelf optimization softwares. In brief, an SOS-1 constraint is a set of variables for which at most one variable in the set may be non-zero. Our new formulation is cleaner because this type of constraints allows us to formulate the feasibility check problem without introducing complicated change of variables as was previously done in \cite{Harirchi2015Model}. Moreover, SOS-1 constraints, which are by nature integral constraints,  make the branch and bound search procedures noticeably faster (see, e.g., \cite[Section 3.3.4]{Gueret1999Applications} for a discussion). Our new model invalidation problem using SOS-1 constraints is presented below, which we will demonstrate to be much faster than an earlier formulation \cite{Harirchi2015Model,Harirchi2016Guaranteed} in Section \ref{sec:runtime}. 
\begin{proposition}\label{prop1}
Given an SWA model $\G$ and an input-output sequence $\big \{ \mathbf{u}_t,{\mathbf{y}}_t\big \}_{t=0}^{N-1}$, the model is invalidated if and only if the following problem is infeasible.
\small{\begin{equation} \tag{P$_{MI}$}
\hspace*{-0.3cm}\begin{aligned} \label{eqn:firstmilp}
	\text{Find } & \x_t, \ETA_t, \NU_t, a_{i,t}, \mathbf{s}_{i,t}, \mathbf{r}_{i,t} \text{ for }  \forall t\in\mathbb{Z}_{N-1}^0, \forall i\in\mathbb{Z}_m^+  \\
	\text{s.t. } & \forall j\in \mathbb{Z}_n^+,\forall k\in \mathbb{Z}_{n_y}^+,\forall l\in \mathbb{Z}_{n_p}^+, \text{ } \forall t\in\mathbb{Z}_{N-1}^0, \text{ we have: } \\ 
	& \x_{t+1} = \A_i \x_t + \B_i \uu_t + \f_i + \NU_t + \mathbf{s}_{i,t}, \\
	& \yy_t = \C_{i} \x_t + \mathbf{D}_{i} \uu_t + \g_{i} + \ETA_t + \mathbf{r}_{i,t}, \\
	& P \x_t \leq p, \;  a_{i,t} \in \{0,1\}, \; \textstyle \sum_{i\in \mathbb{Z}_{m}^+} a_{i,t} = 1, \; \| \NU_{t} \| \leq \epsilon_{\nu}, \\
	&   \| \ETA_t \| \leq \epsilon_{\eta}, \; (a_{i,t}, \mathbf{s}_{i,t}^j): \text{SOS-1}, \; (a_{i,t}, \mathbf{r}_{i,t}^k): \text{SOS-1}, 
\end{aligned}\hspace*{-0.6cm}
\end{equation}} \normalsize
\hspace{-0.15cm}where $\mathbf{s}_{i,t}$ and $\mathbf{r}_{i,t}$ are slack variables that are free when $a_{i,t}$ is zero and zero otherwise.  We refer to this problem as $Feas(\{\uu_t,\yy_t\}_{t=0}^{N-1}, \G)$.
\end{proposition} 

Intuitively, the infeasibility of \eqref{eqn:firstmilp} indicates that there are no state, input and noise values that can generate the input-output sequence from the model, and hence it is impossible that the data is generated by the model. Proposition \ref{Pro:MIP} enables us to solve the model invalidation problem by checking the feasibility of \eqref{eqn:firstmilp}, which is a MILP with SOS-1 constraints that can be efficiently solved with many off-the-shelf softwares, e.g., \cite{gurobi,cplex}. 

\subsection{$T$-Distinguishability} \label{sec:t-detectable}
The model invalidation problem can be solved for the input-output sequence of any given time horizon to detect  faults, but the number of variables and constraints increase with the size of the time horizon. Thus, a few questions naturally arise with regards to this time horizon. 

First, one may ask if the smallest receding time horizon $T$ can be found, for which two different models are guaranteed to be distinguishable. This question leads us to define the notion of $T$-distinguishability, previously presented in \cite{Harirchi2015Model}\footnote{When the pair of models consists of the nominal system model and the fault model, this is also known as $T$-detectability \cite{Harirchi2015Model,Harirchi2016Guaranteed}.}. $T$-distinguishability is defined for a pair of system and/or fault models, which means that the trajectory generated from the two models cannot be identical for a time horizon of length $T$ for any initial state and any noise signals. This notion is very similar to the concept of input-distinguishability, which is defined for linear time-invariant models in \cite{Lou2009Distinguishability,Rosa2011Distinguishability}. $T$-distinguishability is formally defined as follows: 
\begin{definition}[$T$-distinguishability] \label{def:T_detec} A pair of switched affine models $\G$ and $\bar{\G}$ 
is called \emph{$T$-distinguishable} if $\mathcal{B}^T_{swa}(\G) \cap \mathcal{B}^T_{swa}(\bar{\G}) = \emptyset$, where $T$ is a positive integer. 
\end{definition}  

Thus, given two SWA models and an integer $T$, the $T$-distinguishability problem is to check whether the two models are  $T$-distinguishable or not. This problem can be  addressed using a Satisfiability Modulo Theory approach\cite{Harirchi2015Model}, or a MILP feasibility check \cite{Harirchi2016Guaranteed}. 
As with model invalidation in the previous section, we will also propose an alternative MILP formulation for checking $T$-distinguishability, which employs SOS-1 type constraints and as before, is more intuitive and computationally superior (cf. Section \ref{sec:runtime}). Note that in the following $T$-distinguishability test, we have added a decision variable $\delta$ that will be important in a later discussion, which can be computed with little additional computational cost.

%

\begin{theorem} \label{thm:Tdetect}
	A pair of switched affine models $\G$ and $\bar{\G}$ 
	is $T$-distinguishable 
	if and only if the following 
	is infeasible.
{ \small	\begin{equation} \label{eqn:Tdetectfeas} \tag{P$_{T}$}
\hspace*{-0.2cm}\begin{aligned}
&\bar{\delta} = \min_{\x, \bar{\x}, \uu, \ETA, \bar{\ETA}, \NU, \bar{\NU},\mathbf{s},\mathbf{\bar{s}}, \mathbf{r}, a, \delta}  \delta \\
& \text{s. t. }  \forall t \in \mathbb{Z}_{T-1}^0, \forall i \in \mathbb{Z}_m^+, \; \forall j \in \mathbb{Z}_{\bar{m}}^+, \forall k \in \mathbb{Z}_n^+, \forall l \in \mathbb{Z}_{n_y}^+,  \\
& \quad \quad \forall h \in \mathbb{Z}_{n_p}^+, \bar{h} \in \mathbb{Z}_{n_{\bar{p}}}^+,\\
& \quad \quad \x_{t+1} = \A_i \x_t+\B_i\uu_t + \f_i+ \NU_t+ \mathbf{s}_{i,t}, \\
& \quad \quad \bar{\x}_{t+1} = \bar{\A}_j \bar{\x}_t+ \bar{\B}_j\uu_t + \bar{\f}_j+ \bar{\NU}_t+ \mathbf{\bar{s}}_{j,t}, \\
& \quad \quad \mathbf{P}\x_t \leq \mathbf{p}, \; \mathbf{\bar{P}}\bar{\x}_t \leq \mathbf{\bar{p}},   \\ 
& \quad \quad \C_i\x_t \hspace{-0.05cm} + \hspace{-0.05cm} \D_i \uu_t \hspace{-0.05cm} + \hspace{-0.05cm} \g_i \hspace{-0.05cm}+\hspace{-0.05cm} \ETA_t \hspace{-0.05cm}=\hspace{-0.05cm} \bar{\C}_j\bar{\x}_t \hspace{-0.05cm}+\hspace{-0.05cm} \bar{\D}_j\uu_t \hspace{-0.05cm}+\hspace{-0.05cm} \bar{\g}_j \hspace{-0.05cm}+\hspace{-0.05cm} \bar{\ETA}_t \hspace{-0.05cm} + \hspace{-0.05cm}\mathbf{r}_{i,j,t},  \\
&  \quad \quad a_{i,j,t} \in \{0,1\}, \; \textstyle \sum_{i\in \mathbb{Z}_m^+}\sum_{j\in \mathbb{Z}_{\bar{m}}^+}a_{i,j,t} = 1,\\
& \quad \quad \|\ETA_t \| \leq \epsilon_{\eta}, \; \|\bar{\ETA}_t \| \leq \epsilon_{\bar{\eta}}, \; \|\NU_t \| \leq \epsilon_{\nu}, \; \|\bar{\NU}_t \| \leq \epsilon_{\bar{\nu}},\; \| \uu_t \| \leq U,\\
&\quad \quad (a_{i,j,t},\mathbf{s}_{i,t}^{k}): \text{SOS-1}, \; (a_{i,j,t},\mathbf{\bar{s}}_{j,t}^{k}):\text{SOS-1},\\
&\quad \quad (a_{i,j,t},\mathbf{r}_{i,j,t}^{l}): \text{SOS-1},  \;  \left\| \begin{bmatrix}
\ETA_t\\
\NU_t
\end{bmatrix} - \begin{bmatrix}
\bar{\ETA}_t\\
\bar{\NU}_t
\end{bmatrix} \right\| \leq \delta .
\end{aligned}\hspace*{-0.6cm}
	\end{equation}
 }\normalsize We refer to the above-mentioned problem as $Feas_T(\G,\bar{\G})$.
\end{theorem}
\begin{proof}
Except for the last constraint, this is an equivalent formulation to the MILP feasibility problem of Theorem 1 in \cite{Harirchi2016Guaranteed}. Clearly, the last constraint does not change the feasible set, therefore the feasibility of \eqref{eqn:Tdetectfeas} is necessary and sufficient for $T$-distinguishability.
\end{proof}

While Theorem \ref{thm:Tdetect} enables us to solve the $T$-distinguishability problem, if the two models are not $T$-distinguishable, i.e., the solution to \ref{eqn:Tdetectfeas} is feasible, it additionally delivers $\bar{\delta}$, which we argue is a good indication and measure for the separability of two models. In essence, $\bar{\delta}$ can be interpreted as the noise effort that is required to make the trajectories of the two models identical. A larger value for $\bar{\delta}$ indicates a larger separation between the two models that the noise has to compensate for. Hence, we will refer to the normalized version of $\bar{\delta}$ as the \emph{distinguishability index} 
\begin{align}
\delta^*=\frac{\bar{\delta}}{\delta^{\max}},
\end{align}
where $\delta^{\max}\doteq \min\{\max\{\epsilon_{\eta}+\epsilon_{\bar{\eta}},\epsilon_{\nu}+\epsilon_{\bar{\nu}}\},\max\{\epsilon_\eta,\epsilon_\nu\}+\max\{\epsilon_{\bar{\eta}},\epsilon_{\bar{\nu}}\}\}$ is an upper bound on $\bar{\delta}$;  hence, $0 \leq \delta^* \leq 1$. 

Moreover, to find the smallest $T$ for which we have $T$-distinguishability, one could  iterate with $T$ increasing from 1 until the $T$-distinguishability problem in Theorem \ref{thm:Tdetect} becomes infeasible. But, if $\delta^*$ is small for some $T$, then it may make sense to consider larger increases in $T$ to speed up computation. Thus, $\delta^*$ can be used as a heuristic for choosing the next $T$ to solve the $T$-distinguishability problem. In addition, one may also ask about when the iterations with increasing $T$ can be terminated with some confidence that a finite $T$ does not exist. 
Once again, we can consider the trend of $\delta^*$ with increasing $T$ and terminate the iterations when $\delta^*$ reaches a plateau. This will be demonstrated to be effective in a simulation example in Section \ref{sec:distinguishability}. In addition, when this index reaches a (non-zero) plateau and the problem remains not $T$-distinguishable, 
then it would be possible to use any value that is smaller than the maximum $\delta^*$ to derive the maximum allowed uncertainty for a system such that fault detection is guaranteed. This may suggest possible design remedies involving the choice of sensors with better precision or the employment of noise isolation platforms to reduce the amount of noise, in order to facilitate fault detection.

\section{Fault Isolation} \label{sec:fault}
In practice, it is of special interest to identify the source of a detected fault, as it significantly simplify the process of fault accommodation. The process of distinguishing among possible fault models is called fault isolation. In this section, we utilize the tools we developed for distinguishability and model invalidation to address the fault isolation problem for switched affine models, i.e, to  derive necessary and sufficient conditions to guarantee the isolation of a set of possible faults. In addition, we propose a tractable way to isolate faults in real-time for many applications by leveraging the recent advances in the development of mixed-integer linear programming tools \cite{cplex}. In order to ensure the isolability of two fault models, it suffices that the two fault models are $T$-distinguishable. In this case, by implementing model invalidation for each of the two models on a time horizon of length $T$, we can simply isolate the faults after detection. Hence, we turn our attention to the multiple-fault scenario.

\subsection{Multiple-Fault Scenario}
The fault isolation problem becomes marginally more challenging 
with multiple fault models. Let us assume that there exist $N_f$ fault models for a specific system. We consider the following necessary and sufficient assumptions:
\begin{assumption} 
[Detectability Assumption] \label{a:detectable} We assume that $\forall \, j \in \mathbb{Z}_{N_f}^+$, there exists a finite $T_j$ such that the pair of nominal system $\G$ and the fault model $\bar{\G}_j$ is $T_j$-distinguishable. 
\end{assumption}
\begin{assumption}
[Isolability Assumption] \label{a:isolable} We assume that $\forall \, m,n\in \mathbb{Z}_{N_f}^+$, $m \neq n$, there exists a finite $I_{m,n}$ such that $\bar{\G}_m$ and $\bar{\G}_n$ are $I_{m,n}$-distinguishable.
\end{assumption}

Now, we will define $I$-isolability for multiple faults.
\begin{proposition}
	[$I$-isolability for multiple faults] Consider a set of $N_f$ fault models that satisfies Assumption \ref{a:isolable}. If a fault occurs, it can be isolated in at most $I = \max_{m,n, \; m\neq n} I_{m,n}$ steps after the occurrence. Such a set of fault models is called $I$-isolable. 
\end{proposition}
\begin{proof}
Under Assumption \textit{A}\ref{a:isolable}, and because $I \geq I_{m,n}$ for all possible pairs of fault models, all pairs of faults are $I$-distinguishable. Therefore, if any of the faults occur persistently, by observing at most $I$ samples, it will be isolated. This is because the length-$I$ behavior of the occurred fault does not have any intersection with the length-$I$ behavior of any of the other faults.
\end{proof}

\section{FDI Scheme} \label{sec:FDI}
In this section, we propose a two-step FDI scheme: \\[-2.pt]
\subsubsection{Off-line step} In the off-line step, under Assumptions $A$\ref{a:detectable} and $A$\ref{a:isolable}, we calculate the following quantities:
	{\small\begin{equation*}
		\begin{aligned}
			& \text{Isolability index: } I = \max_{m,n} I_{m,n}, \; m,n \in \mathbb{Z}_{N_f}^+, \; m \neq n,\\
			& \text{Isolability index for fault $i$: } \tilde{I}_i = \textstyle \max_{j\in \mathbb{Z}^+_{N_f},\; j \neq i} I_{i,j},\\
			& \text{Detectability index: } T = \textstyle \max_{j\in \mathbb{Z}^+_{N_f}} T_j.\\
			& \text{Length of memory: } K = \max \{T,I\}
		\end{aligned}
	\end{equation*}}
\vspace{-0.35cm}
\subsubsection{On-line step} In this step, we leverage $N_f+1$ \emph{parallel} monitors corresponding to system and fault models.  The monitors are labeled as $\{\mathcal{M}_0, \mathcal{M}_1, \hdots, \mathcal{M}_{N_f} \}$, where $\mathcal{M}_0$ corresponds to the system model and $\mathcal{M}_i$ corresponds to the $i^{th}$ fault model. First, only $\mathcal{M}_0$ is active for fault detection. The rest of the monitors will be ``off'' until a fault is detected by $\mathcal{M}_0$. The inputs to each monitor at time $t$ are the input-output sequence of length $K_i=\max\{\tilde{I}_i,T_i\}$, $\{\uu_k,\yy_k \}_{k = t-K_i+1}^{t}$, and the corresponding model $\bar{G}_i$. For instance, $\mathcal{M}_0$ knows $\G$, and at each time step, it solves the model invalidation problem, $Feas(\{\uu_k,\yy_k \}_{k = t-T+1}^{t}, \G)$. If the problem is feasible, the monitor outputs $0$, otherwise it outputs $1$. In the latter case, the bank of fault monitors is activated and \emph{parallelly} solves the model invalidation problems for all fault models, i.e., to check if $\mathcal{M}_j$ solves $Feas(\{\uu_k,\yy_k \}_{k = t-K_j+1}^{t}, \bar{\G}_j)$ for each $j \in \mathbb{Z}_{N_f}^+$. By Assumptions $A$\ref{a:detectable} and $A$\ref{a:isolable}, it is guaranteed that in this case, the problem of at most one monitor is feasible. The output block receives the signal from all the monitors and shows two elements, the first element is 1, which indicates that a fault has occurred, and the second element is $k_f \in \mathbb{Z}_{N_f}^+$ if the fault matches $k_f$$^{th}$ fault model, or 0 if the fault does not match any of the fault models.  


At every time step $t$, this FDI scheme acts as a function:
\begin{equation}
[\mathcal{H}, \mathcal{F}] = \psi(\{ \uu_k, \yy_k\}_{k= t-K+1}^t, \G, \{ \bar{\G}_j\}_{j=1}^{N_f}),
\end{equation}
where $\mathcal{H}$ is 
0 or 1 to indicate healthy or faulty behaviors, and $\mathcal{F}$ either indicates the fault model that is active, or that none of the fault models caused the faulty behavior. 
\begin{remark}
	In some practical examples, Assumptions $A$\ref{a:detectable} and $A$\ref{a:isolable} may not be satisfied, i.e., the FDI approach is not guaranteed to detect and isolate the given fault models. However, the FDI approach can be simply modified such that $\mathcal{F}$ outputs either the set of faults that matches the data (because some fault models may not be isolable) along with their corresponding `likelihoods' in terms of their distinguishability indices, or the empty set, if none of the models matches the data. 
\end{remark}

\subsection{Detection and Isolation Delays}
In this section, we describe the notion of delays in detection and isolation of faults, and provide theoretical bounds on these delays using detectability and isolability indices.
\begin{definition}
	(Detection/Isolation Delay) Detection/isolation delay is the number of time samples it takes from the occurrence of the fault to its detection/isolation. We denote detection and isolation delays with $\tau_{T}$ and  $\tau_{I}$, respectively. 
\end{definition} 
\begin{proposition}
	Given a $T_i$-distinguishable pair of system and fault models $(\G,\bar{\G}_i)$, the detection delay of the proposed fault detection scheme is bounded by $T_i$. In addition, the isolation delay of a pair of $I_{i,j}$-isolable fault models $(\bar{G}_i,\bar{G}_j)$ is bounded by $I_{i,j}$.
\end{proposition}
\begin{proof}
Direct consequence of definitions.	
\end{proof}
\begin{theorem}
	The detection delay for fault $\bar{\G}_i$ using the FDI scheme proposed in Section \ref{sec:FDI} is bounded by $T_i$, and the isolation delay is bounded by $K_i=\max\{\tilde{I}_i,T_i\}$.
\end{theorem}
\begin{proof}
Assume fault $i$ occurs at time $t^*$. The FDI approach implements model invalidation with a time horizon size of $T\geq T_i$. At time $t^*+T_i-1$, the input-output trajectory that is fed to the model invalidation contains a length $T_i$ trajectory that is in $\mathcal{B}^{T_i}(\bar{\G}_i)$. By $T_i$-distinguishability of $\bar{\G}_i$, this trajectory cannot be generated by $\G$. Thus, the model will be invalidated by observing at most $T_i$ data points from fault $i$. This concludes the proof for the bound on detection delay. For isolation, the FDI approach requires detection first, and in the worst case, detection will occur in $T_i$ steps. On the other hand, if we observe any trajectory from $t^*$ to $t^*+\tilde{I}_i-1$ that is generated by fault $i$, it is not in $\mathcal{B}^{\tilde{I}_i}(\bar{\G}_j), j\neq i$. This is because $\tilde{I}_i \geq I_{i,j}, j\neq i$. Hence, the fault will be isolated in at most $\tilde{I}_i$ observations of the fault. Considering that the fault needs to be detected first, the isolation delay is bounded by $K_i=\max\{\tilde{I}_i,T_i\}$. This concludes the proof.  
\end{proof} 

\subsection{Adaptive Fault Isolation}
The bound on 
isolation delays represents the worst case scenario, where the data created by a fault model falls 
within the behavior of some other models up until the very last time step. However, this is not the case in most applications, where the faults can be 
isolated much prior to this bound. Here, in this section, we propose an adaptive fault isolation scheme that reduces isolation delay, which is based on the idea of validation of only one of the fault models. 
Since the data prior to the time of detection will most probably invalidate all the fault models, we propose 
to reduce isolation delays by using an adaptive receding horizon that considers only the data starting from the detection time (fixed horizon lower bound) with increasing horizon until only one fault model matches or validates the data. In practice, we can achieve this by considering model invalidation problems for each of the fault models with the adaptive receding horizon until only one fault model remains that matches the data. 

Since we assumed that the fault is among the predefined set of models and is persistent, it is guaranteed that the fault will be isolated with this approach. Such an approach has the potential to significantly reduce isolation delays, as we have observed in simulation in Section \ref{sec:HVAC} (cf. Fig. \ref{fig:results} (bottom)).

\section{Illustrative Examples} \label{sec:example}

First, we demonstrate in Section \ref{sec:runtime} that our new formulations for model invalidation and $T$-distinguishability in Prop. \ref{prop1} and Thm. \ref{thm:Tdetect}, respectively, are computationally superior to the previous formulation in  \cite{Harirchi2015Model,Harirchi2016Guaranteed}.
Then, we illustrate the performance of the proposed FDI scheme using a numerical model for the Heating, Ventilating, and Air Conditioning (HVAC) system that is proposed in \cite{Arguello1999Nonlinear} in Section \ref{sec:HVAC}. Moreover, we provide a numerical example in Section \ref{sec:distinguishability} to illustrate the practical merits of the distinguishability index that was introduced in Section \ref{sec:t-detectable}. All the simulations in this section are implemented on a 3.5 GHz machine with 32 GB of memory that runs Ubuntu. For the implementation of the MILP feasibility check problems, we utilized YALMIP \cite{YALMIP} and Gurobi \cite{gurobi}. All the approaches and examples are implemented in MATLAB.

\subsection{Run-Time Comparison} \label{sec:runtime}
In this section, we compare the run-time for the formulations proposed in this paper with the one 
in \cite{Harirchi2016Guaranteed}. Consider a hidden-mode switched affine model, 
$\G$, with admissible sets $\X = \{ \x \mid\| \x \| \leq 11 \}$, $\U = \{ \uu \mid\| 
\uu \| \leq 1000 \}$ and $\E = \{ \ETA \mid\| \ETA \| \leq 0.1 \}$. We assume there is no process noise. We also assume 
$B =[1 \;\; 0 \;\; 1]^\intercal$ and $C = [1 \;\; 1 \;\; 1]$ for all modes. The system matrices of the modes are:

\vspace{-0cm}
{\scriptsize 
	\begin{align*} 
	& \A_1 \hspace{-0.05cm}=\hspace{-0.05cm} \begin{bmatrix}
	0.5 & 0.5 & 0.5\\
	0.1 & -0.2 & 0.5\\
	-0.4 & 0.6 & 0.2
	\end{bmatrix}\hspace{-0.1cm}, \f_1 \hspace{-0.05cm}=\hspace{-0.05cm} \begin{bmatrix}
	1 \\ 0 \\ 0
	\end{bmatrix}\hspace{-0.1cm},  \A_2 \hspace{-0.05cm}=\hspace{-0.05cm} \begin{bmatrix}
	0.5 & 0.5 & 0.5\\
	-0.3 & -0.2 & 0.3 \\
	0.1 & -0.3 & -0.5
	\end{bmatrix}\hspace{-0.1cm},  \f_2 \hspace{-0.05cm}=\hspace{-0.05cm} \begin{bmatrix}
	0 \\ 1 \\ 0
	\end{bmatrix}\hspace{-0.1cm}, \nonumber \\
	& \A_3 \hspace{-0.05cm}=\hspace{-0.05cm} \begin{bmatrix}
	0.5 & 0.2 & 0.6\\
	0.2 & -0.2 & 0.2\\
	-0.9 & 0.7 & 0.1
	\end{bmatrix}\hspace{-0.1cm},  \f_3 \hspace{-0.05cm}=\hspace{-0.05cm} \begin{bmatrix}
	0 \\ 0 \\ 1
	\end{bmatrix}\hspace{-0.1cm}.
	\end{align*} 
}

In addition, consider a fault model, $\G^f$, with: 

{\scriptsize
	\begin{align*} \label{eqn:numexfault}
	& \A^f = \begin{bmatrix}
	0.8 & 0.7 & 0.6\\
	0.1 & -0.2 & 0.3\\
	-0.4 & 0.3 & -0.2
	\end{bmatrix}\hspace{-0.05cm}, \ \B^f = \begin{bmatrix}
	1 \\ 0 \\ 0
	\end{bmatrix}\hspace{-0.05cm}, \ \f^f = \begin{bmatrix}
	1 \\ 1 \\ 1
	\end{bmatrix}.
	\end{align*}}
	
The implementation of the $T$-distinguishability approach proves that the system and fault model pairs is 12-distinguishable. We first randomly generate input-output trajectories (5 for each time horizon length) from $\G^f$. 
We then compare the model invalidation approaches 
that use the proposed formulation in Prop. \ref{prop1} and the one in \cite{Harirchi2015Model,Harirchi2016Guaranteed}. The average run-time for each time horizon length as well as the standard deviation of run-times for both formulations are illustrated in Fig. \ref{fig:run_time}. Clearly, the results indicate the superiority of the proposed formulation to the one in \cite{Harirchi2015Model,Harirchi2016Guaranteed}. Similar improvements were also observed for the proposed $T$-distinguishability formulation in Thm. \ref{thm:Tdetect} when compared to \cite{Harirchi2015Model,Harirchi2016Guaranteed} (plots are omitted for brevity).

\begin{figure}[h!] 
	\centering
	\includegraphics[width=2.1in,trim=0mm 5mm 0mm 0mm]{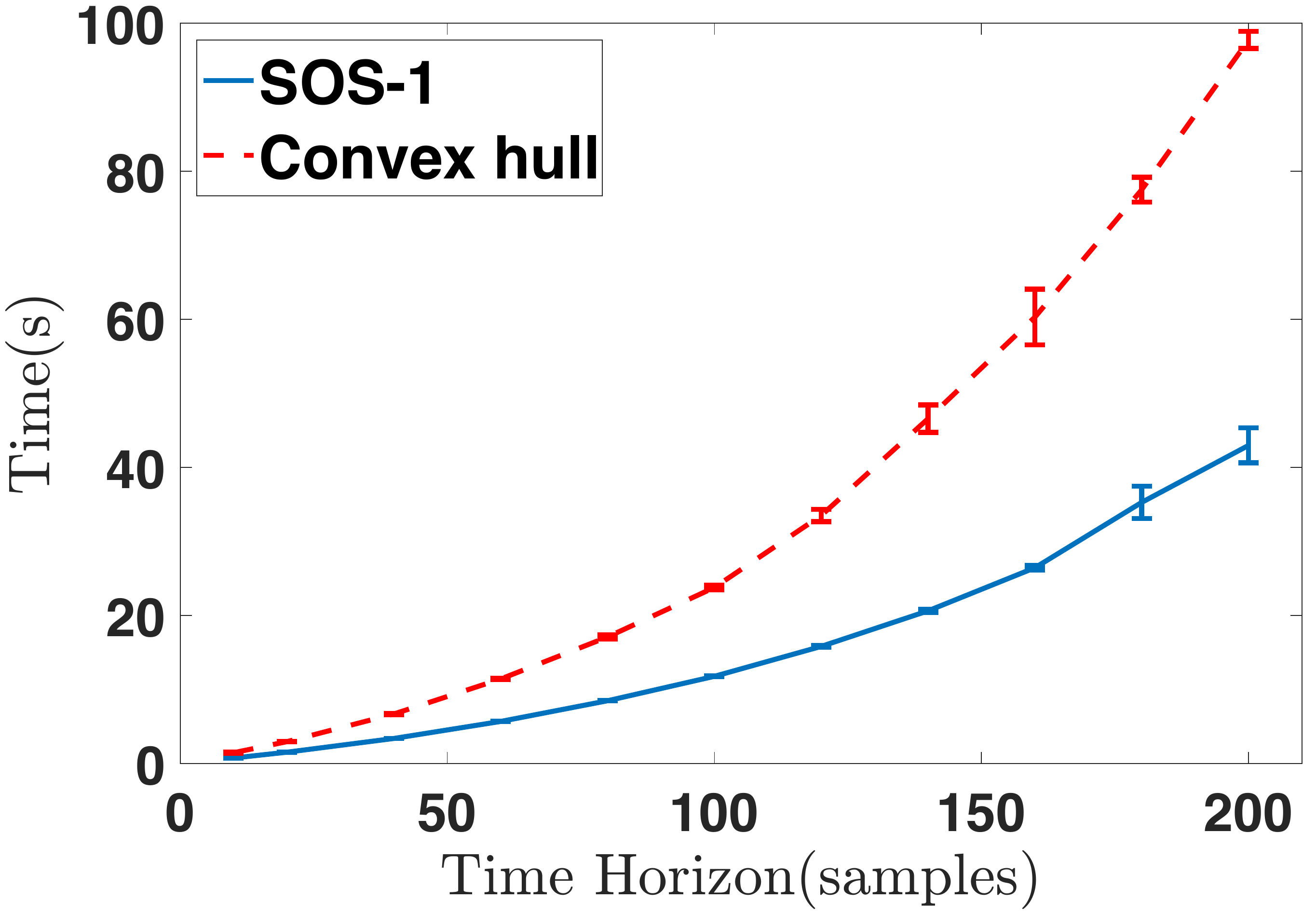}
	\caption{Average execution time (with standard deviations) for invalidating data generated by $\G^f$ 
	with various time horizons.} \vspace{-0.1cm}
	\label{fig:run_time}
\end{figure}

 \vspace{-0.2cm}
\subsection{Fault Diagnosis in HVAC Systems}\label{sec:HVAC}
\begin{figure}[h!]
	\centering
	\includegraphics[width=1.7in]{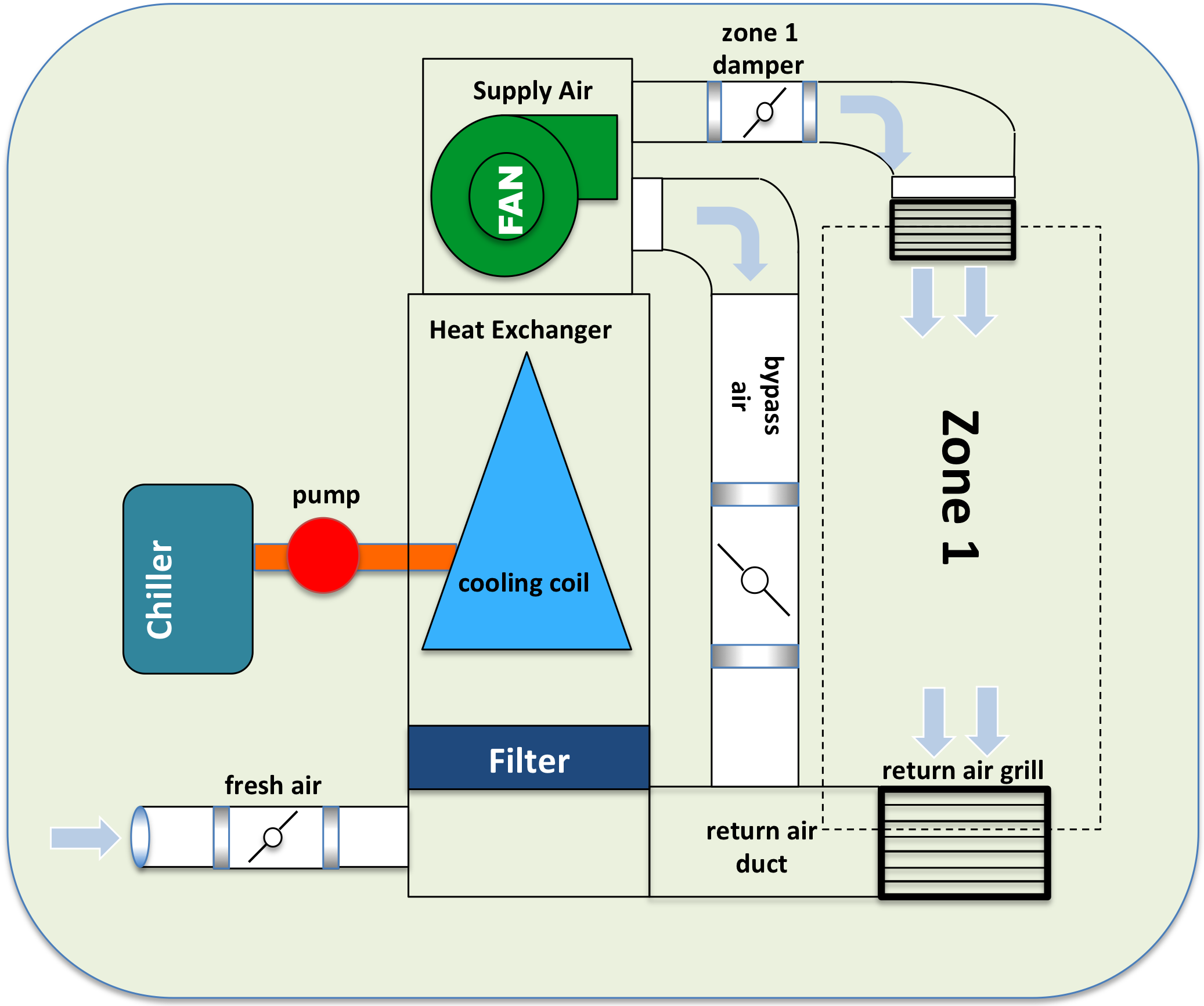}
	\caption{Schematic of a single-zone HVAC system.}  \vspace{-0.1cm}\label{fig:hvacscheme}
\end{figure}

%
%

\begin{figure*}[t!]
\begin{center}
	\begin{minipage}[right]{2.1in}
		\includegraphics[width=2.1in,trim=0mm 12mm 0mm 0mm]{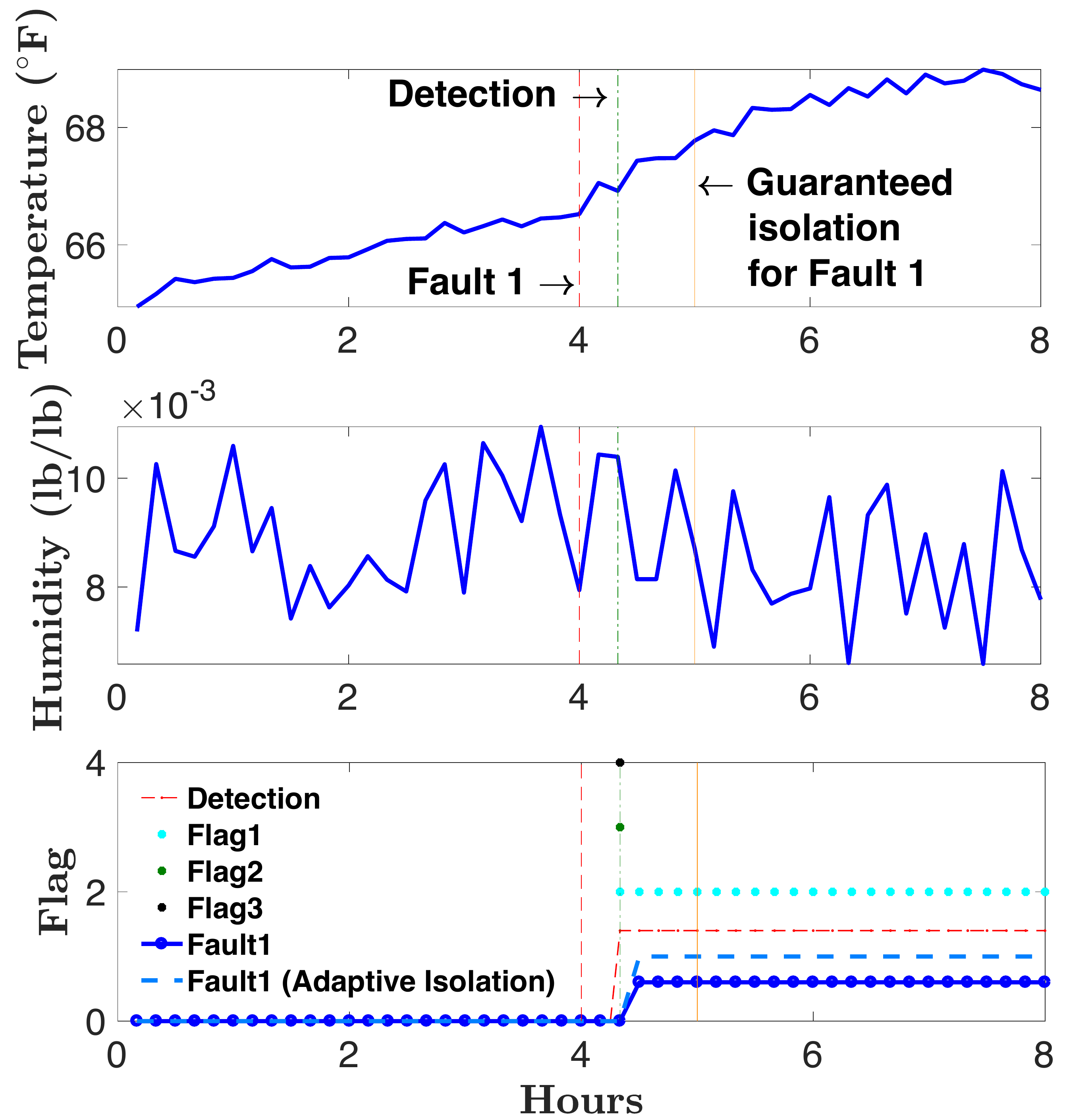}
	\end{minipage} \quad \
	\begin{minipage}[right]{2.1in}
		\includegraphics[width=2.1in,trim=0mm 12mm 0mm 0mm]{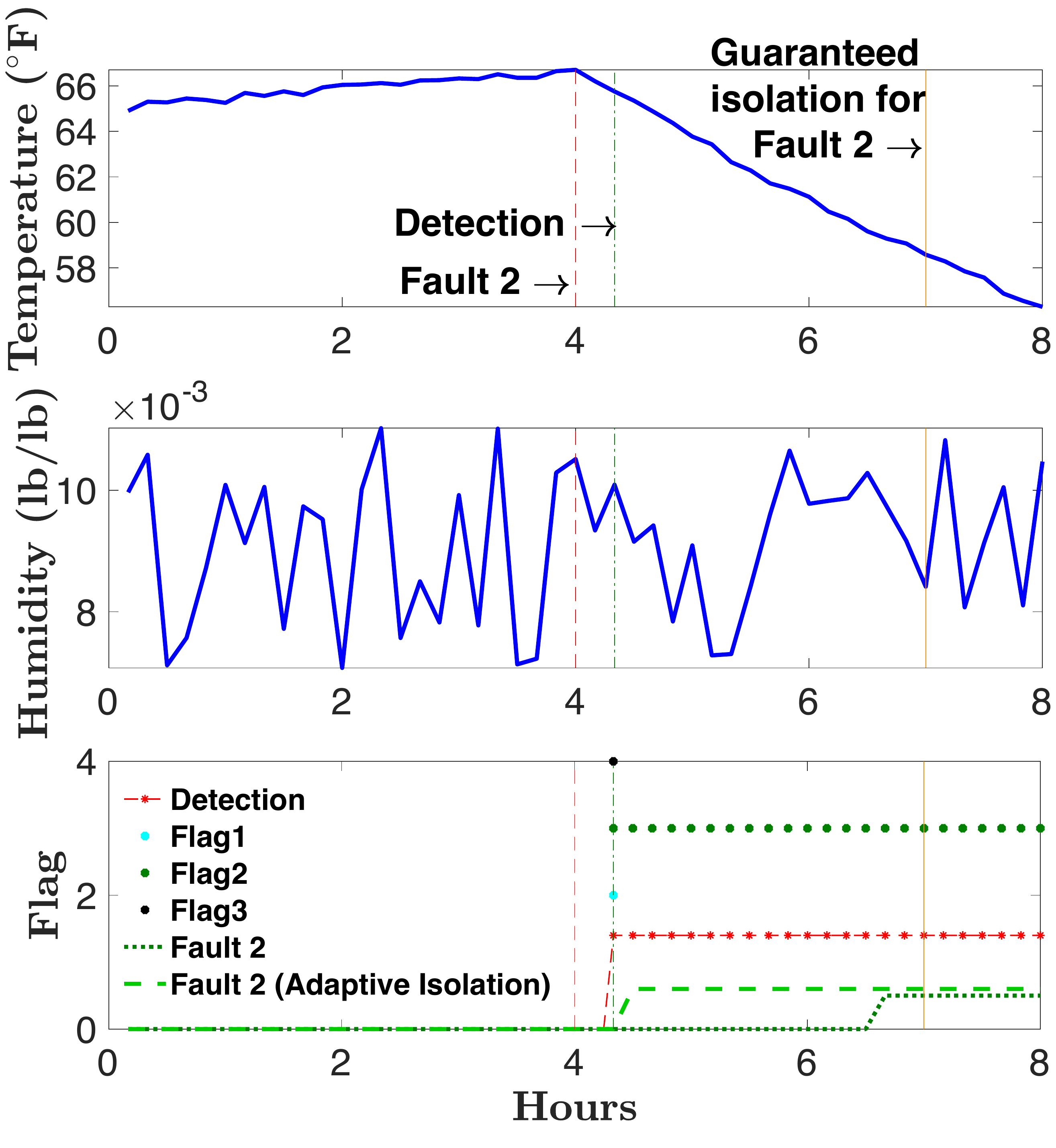}
	\end{minipage} \quad \
	\begin{minipage}[right]{2.1in}
		\includegraphics[width=2.1in,trim=0mm 12mm 0mm 0mm]{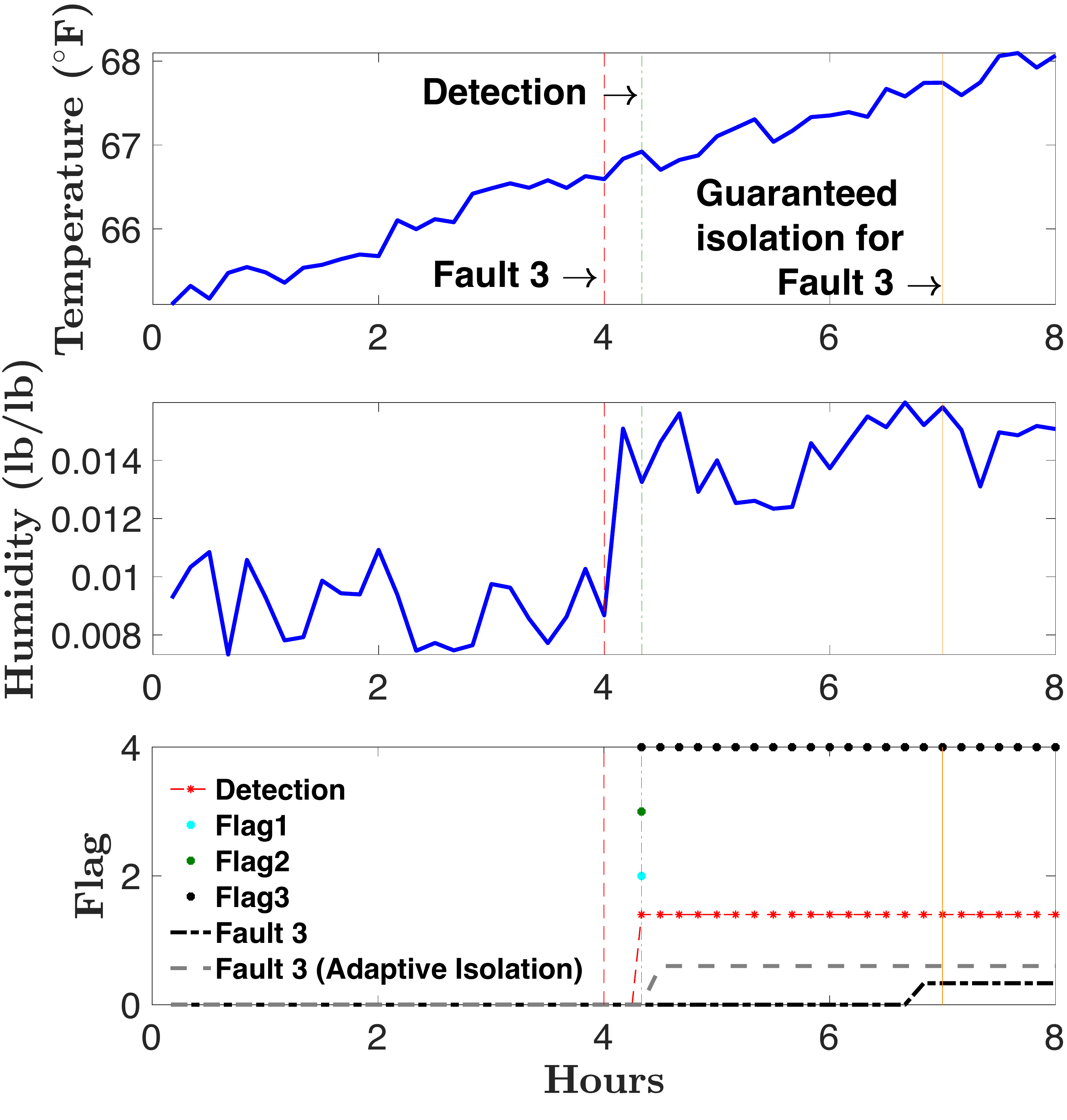}
	\end{minipage}
	\end{center}
	\caption{The outputs (top two rows) of 3 fault scenarios; Detection, isolation and adaptive isolation signals for all faults (bottom row). Flag $i$ is non-zero when the model invalidation problem associated with fault $i$ using the adaptive horizon length is validated. Adaptive isolation occurs when only one Flag is non-zero.}\vspace*{-0.65cm} \label{fig:results}
\end{figure*}
In \cite{Arguello1999Nonlinear}, a single-zone HVAC system in cooling mode (cf. schematic in Fig. \ref{fig:hvacscheme}) is considered. 
This HVAC system is represented by a non-linear model as follows:

\vspace*{-0.15cm}
{\scriptsize \begin{equation} \label{eqn:nonmodel}
\begin{aligned}
\begin{pmatrix}
\dot{T}_{TS} \\
\dot{W}_{TS} \\
\dot{T}_{SA} 
\end{pmatrix} = & \begin{pmatrix}
-\frac{f}{V_s} & \frac{h_{fg}f}{C_pV_s} & \frac{f}{V_s}  \\
0 & -\frac{f}{V_s} & 0 \\
0.75\frac{f}{V_{he}} & -0.75\frac{fh_w}{C_pV_{he}} & - \frac{f}{V_{he}} 
\end{pmatrix}\begin{pmatrix}
{T}_{TS} \\
{W}_{TS} \\
{T}_{SA} 
\end{pmatrix}+ \\  
& \begin{pmatrix}
- \frac{h_{fg}f}{C_pV_s}W_s + \frac{4}{C_pV_s}(Q_o -h_{fg}M_o)\\
\frac{f}{V_s}W_s + \frac{M_o}{\rho V_s}\\
\frac{f}{4V_{he}} (T_o-\frac{h_w}{C_p}W_o) + \frac{fh_w}{C_pV_{he}} W_s - 6000 \frac{gpm}{\rho C_p V_{he}}\\
\end{pmatrix},
\end{aligned}
\end{equation}}
\hspace{-0.135cm}where $f$, $gpm$, $M_o$  and $Q_o$ are time varying parameters. The parameters of the model\footnote{$h_w=180, \; h_{fg}=1078.25, \; W_o=0.018, \; W_s=0.007, \; C_p=0.24, \; T_o=85, \; V_s=58464, \;V_{he}=60.75, \; M_o\in [150 \; \; 180], \; Q_o \in [289800\; \; 289950], \; \rho=0.074, \; f = 17000, \; gpm \in \{0,58\}$.} are defined in \cite{Arguello1999Nonlinear}. 

We leverage an augmented state-space model with additional states $Q_0$ and $M_0$ that is obtained in \cite{Arguello1999Nonlinear}. To further simplify the model, we assume that the fan is always turned on and the flow rate is fixed at $17000$ ft$^3$/min and the chiller pump is either ``off'' or ``on'' with a fixed flow rate of $58$ gal/min. These assumptions along with a discretization with a sampling time of $10$ minutes convert the nonlinear system 
\eqref{eqn:nonmodel} to a switched affine model 
parameterized by

 \vspace{-0.2cm}
{\scriptsize \begin{equation*} 
	\begin{aligned} 
	& A_1 = A_2=\begin{pmatrix}
	0.98 & 229.63 & 0.001 & 0 & -0.0035\\ 
	0 & 0.94 & 0 & 0 & 0\\ 
	0.74 & -360.61 & 0.0008 & 0 & -0.0030\\
	0 & 0 & 0 & 1 & 0\\
	0 & 0 & 0 & 0 & 1
	\end{pmatrix}\hspace{-0.05cm}, \; f_2 \hspace{-0.05cm}=\hspace{-0.05cm} \begin{pmatrix}
	0.3886\\
	0.0001 \\
	-22.576 \\
	0\\
	0
	\end{pmatrix}\hspace{-0.05cm}, \\
	&  C_1 = C_2 =\begin{pmatrix}
	1 & 0 & 0 & 0 & 0 \\
	0 & 1 & 0 & 0 & 0 
	\end{pmatrix}\hspace{-0.05cm}, \;f_1 =\mathbf{0}. 
	\end{aligned}
	\end{equation*}} \vspace{-0.2cm}
	
The first three states in the SWA model represent the deviation of $T_{TS}, W_{TS}$ and $T_{SA}$ from their equilibria and the last two states are $Q_0$ and $M_0$. In addition, the HVAC model is represented by $\G_{H}=(\X,\E,\U,\{G_i\}_{i=1}^2)$, where $\X = \{\x \mid [-100 \; -0.05 \; -50 \; 289800 \; 150]^\intercal \leq \x \leq [100\; 0.05 \; 50 \; 289950 \; 180]^\intercal  \}$, $\E = \{ \ETA \mid |\ETA|\leq [0.2 \; 0.002]^\intercal \}$ and $\U = \emptyset$. The last two bounds on the states are for the augmented states, which are assumed to stay within a small range of their equilibria. The first mode corresponds to chiller being ``on'' and the second mode represents the model when it is ``off''. The controller keeps the temperature in the comfort zone of $65$--$75^\circ F$ by turning the chiller on and off. Control signals are not observed by the FDI scheme.

\begin{figure}[t!]
	\centering
	\includegraphics[width=1.6in,trim=2mm 2.5mm 2mm 3.5mm]{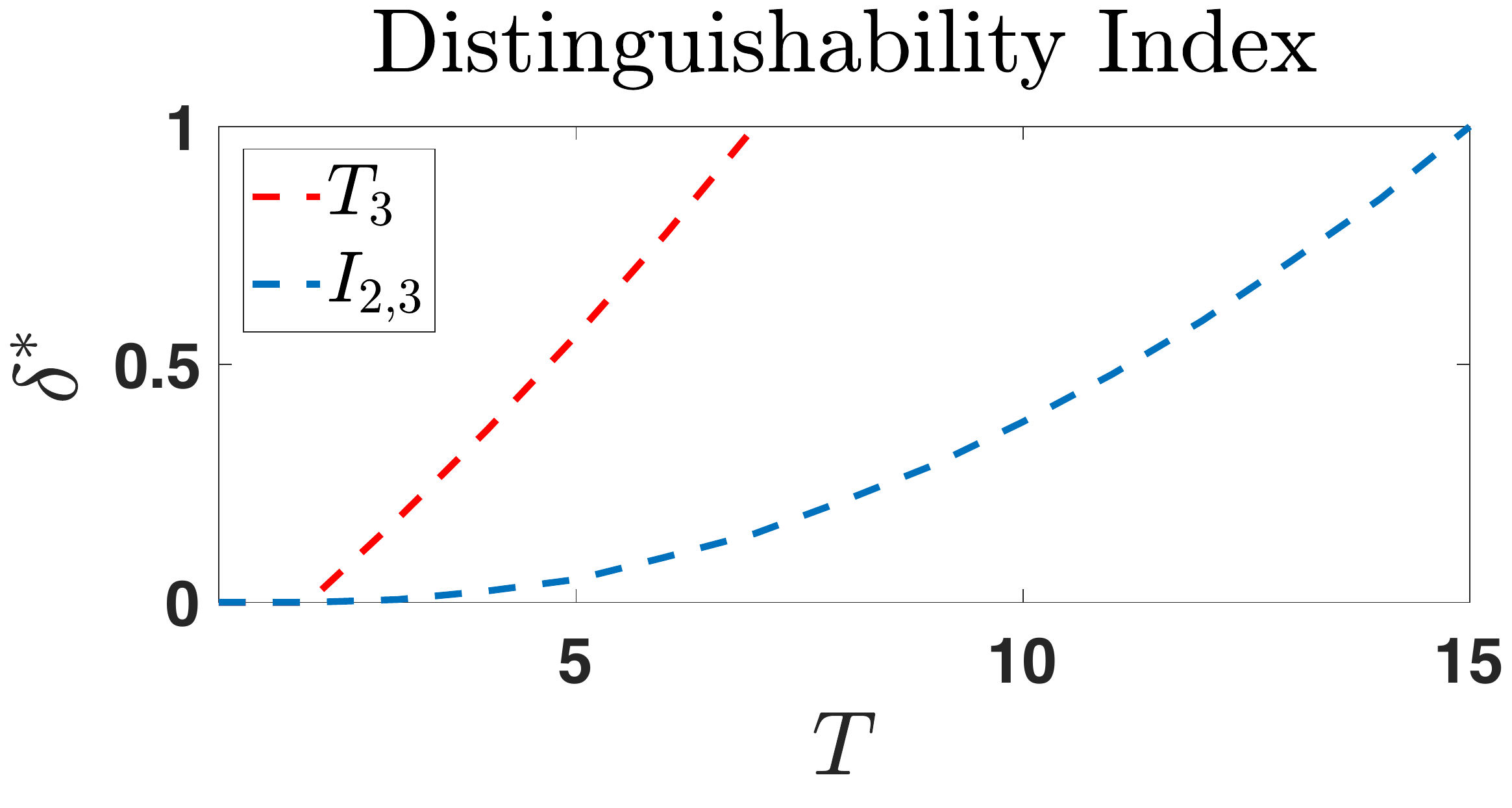} \ \ \
	\includegraphics[width=1.6in,trim=2mm 2.5mm 2mm 3.5mm]{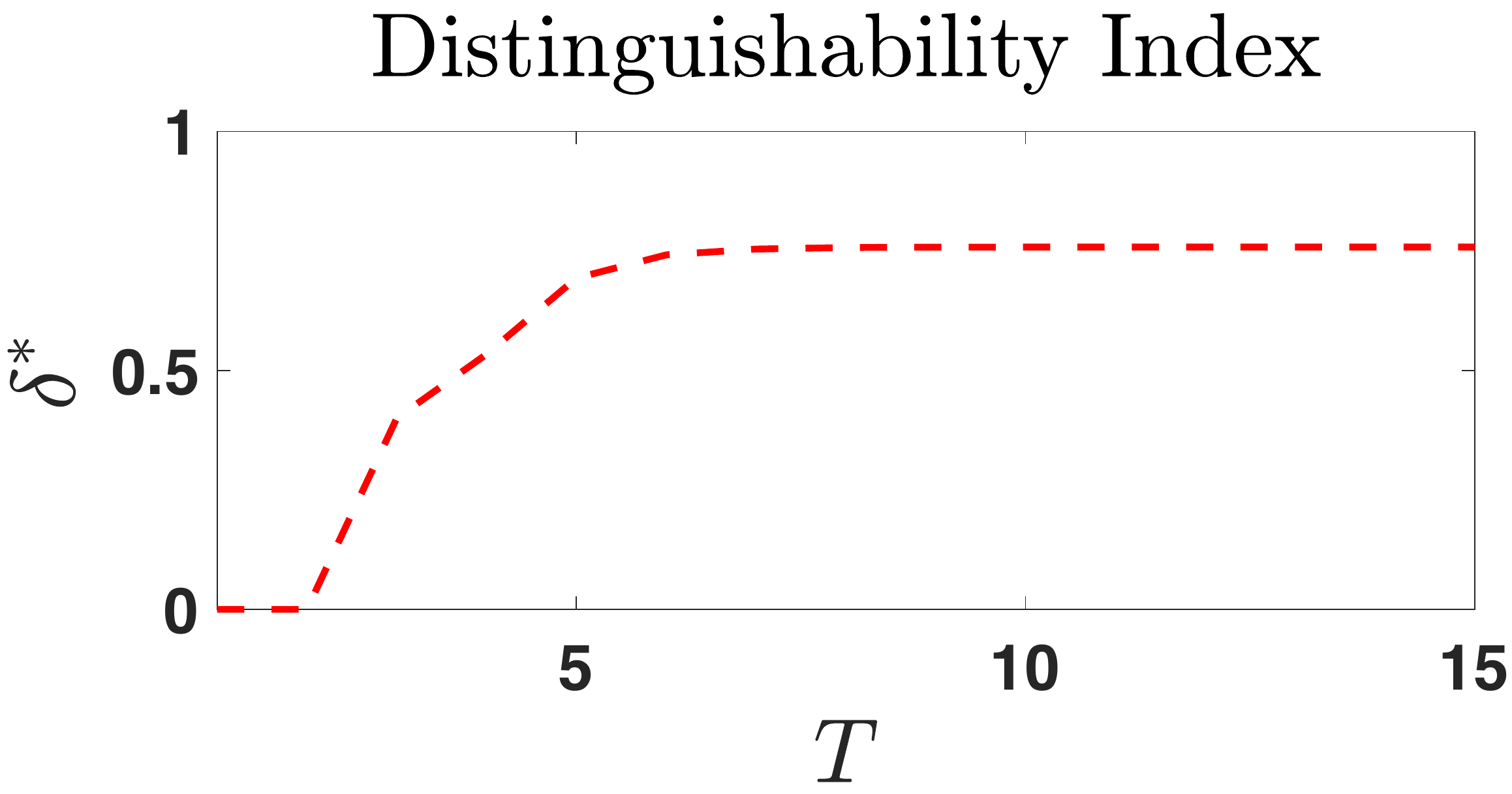}
	\caption{Distinguishability index as a function of length of time horizon. Left: Increase in detectability index for fault 3, $T_3$ and isolability index of faults 2 and 3, $I_{2,3}$, for the HVAC example; Right: Nonlinear increase with a plateau at around $T=5$, for numerical example described by \eqref{eqn:nummodel}.}  \label{fig:distindex}\vspace{-0.15cm}
\end{figure}

We consider three fault scenarios:
\begin{enumerate}
	\item Faulty fan: The fan rotates at half of its nominal speed.
	\item Faulty chiller water pump: The pump is stuck and spins at half of its nominal speed.
	\item Faulty humidity sensor: The humidity measurements are biased by an amount of +0.005. \vspace{-0.05cm}
\end{enumerate}
The proposed approach for $T$-distinguishability and $I$-isolability gives us the following results:

\vspace{-0.1cm}
\begin{table}[h!] \scriptsize
	\centering 
	\caption{Detectability and Isolability Indices} 
	\renewcommand{\arraystretch}{1.1}
	\begin{tabular}{llllll}	\hline
	$T_1 = 4$ & $T_2 = 8$ & $T_3 = 16$ & $I_{1,2}=4$ & $I_{1,3}=4$ & $I_{2,3}=16$ \\ \hline
	\end{tabular}
\end{table}

To illustrate the growth in the distinguishability index $\delta^*$ as the time horizon increases,  we plot its trend  in Fig. \ref{fig:distindex} (left) for $T$-distinguishability of fault 3 and $I$-isolability of faults 2 and 3. The plot shows that the distinguishability index we introduced does indeed deliver a nice measure of how far two models are from detectability or isolability, and at the same time, it allows us to estimate the size of time horizon, $T$ or $I$, to achieve $T$-distinguishability or $I$-Isolability.

Next, we consider 3 scenarios, where for each scenario $i$ ($i\in\{1,2,3\}$), we generate data from the nominal system for four hours and from fault $i$ afterwards. The times at which the faults occur and their detection times, as well as the upper bounds on isolation delays are indicated in Fig. \ref{fig:results} (top and middle rows), which show the output trajectories for each scenario. Furthermore, we plot in Fig. \ref{fig:results} (bottom row) the detection and isolation signals for all three faults to show that only the occurred fault is isolated in all scenarios before their upper bounds are exceeded, and that the proposed adaptive isolation scheme reduces the isolation delay, as desired. 

 \subsection{Distinguishability Index} \label{sec:distinguishability}
To illustrate the practical use of the distinguishability index, $\delta^*$, we consider two synthetic SWA models $\G$ and $\bar{\G}$ subject to measurement and process noise, given by

 \vspace{-0.2cm}
{\scriptsize \begin{equation} \label{eqn:nummodel}
	\begin{aligned}
	& \G: \begin{cases}
		& A_1 =\begin{pmatrix}
		0.1 & 0 & 0.1 \\ 
		0 & 0.1 & 0.2 \\ 
		0.2 & 0.12 & 0 
		\end{pmatrix}, A_2 =\begin{pmatrix}
		0 & 0 & 0.15 \\ 
		0.1 & 0 & 0 \\ 
		0.1 & 0.12 & 0.1 
		\end{pmatrix},\\
		&  C_1 = C_2 = I , \;f_1 =\begin{pmatrix}
		0.5\\
		0.2 \\
		1
		\end{pmatrix}, \; f_2 = \begin{pmatrix}
		1\\
		0 \\
		0.5
		\end{pmatrix},
	\end{cases}\\
	& \bar{\G}: \begin{cases}
	& \bar{A}_1 =\begin{pmatrix}
	0.1 & 0 & 0.1 \\ 
	0 & 0.1 & 0.2 \\ 
	0.2 & 0.1 & 0 
	\end{pmatrix}, \bar{A}_2 =\begin{pmatrix}
	0 & 0 & 0.1 \\ 
	0.1 & 0 & 0 \\ 
	0.1 & 0.1 & 0.1 
	\end{pmatrix},\\
	&  \bar{C}_1 = \bar{C}_2 = I , \;\bar{f}_1 =\begin{pmatrix}
	0.3\\ 
	0 \\
	0.9
	\end{pmatrix}, \; \bar{f}_2 = \begin{pmatrix}
	0.8\\  
	0.2 \\
	0.3
	\end{pmatrix},
	\end{cases}
	\end{aligned}
	\end{equation}}
\hspace{-0.2cm}where the rest of the parameters are zero. The bounds on the process and measurement noise are set to be $0.2$ and $0.25$, respectively.
Fig. \ref{fig:distindex} (right) depicts the change of the distinguishability index with increasing $T$. We observe that the distinguishability index increases nonlinearly and reaches a plateau at a value less than one. In this case, the distinguishability index $\delta^*$ provides a practical indication that these two models are very unlikely 
to be isolable for any finite $I$. Moreover, if the noise levels are constrained to be below the value of the plateau, then we can be sure that these faults will be isolable. Hence, the distinguishability index can also be exploited to derive the maximum allowed uncertainty for a system such that certain faults are guaranteed to be detectable or isolable. In turn, this suggests possible measures for ensuring fault detection and isolation through the reduction of noise levels, either with a better choice of sensors or with the use of noise isolation platforms.


\section{Conclusion} \label{sec:conclude}
We considered the FDI problem for switched affine models. For fault detection, we proposed new model invalidation and $T$-distinguishability formulations using SOS-1 constraints, that are demonstrated to be computationally more efficient and do not require any complicated change of variables. Further, we introduced the \emph{distinguishability index} as a measure of separation between the system and fault models and showed that this index is also a practical tool for finding the smallest receding time horizon that is needed for fault detection and for recommending system design changes for ensuring fault detection. Finally, a novel approach is proposed for isolation of a set of faults, 
with proven isolation guarantees under certain conditions. The effectiveness of the proposed approaches is illustrated on an HVAC system. 
\bibliographystyle{unsrt}
\bibliography{modelinvalidation,fault}

\end{document}